\documentclass[11pt]{article}
\usepackage[T1]{fontenc}
\usepackage{latexsym,amssymb,amsmath,amsfonts,amsthm}
\usepackage{graphics}
\usepackage{graphicx}
\usepackage{mathrsfs}
\usepackage{subfigure}
\usepackage{xcolor}
\usepackage{color}
\newcommand{\R}{{\mathbb R}}
\newcommand{\N}{{\mathbb N}}
\newcommand{\C}{{\mathbb C}}
\newcommand{\be}{\begin{eqnarray}}
\newcommand{\ben}{\begin{eqnarray*}}
\newcommand{\en}{\end{eqnarray}}
\newcommand{\enn}{\end{eqnarray*}}

\newtheorem{theo}{Theorem}[section]
\newtheorem{coro}[theo]{Corollary}

\newtheorem{rema}[theo]{Remark}

\definecolor{rot}{rgb}{0.000, 0.000, 0.000}
\newcommand{\tcr}{\textcolor{rot}}

\definecolor{green}{rgb}{0.000, 0.000, 0.000}
\newcommand{\tcg}{\textcolor{green}}

\definecolor{blue}{rgb}{0.000, 0.000, 0.000}
\newcommand{\tcb}{\textcolor{blue}}

\definecolor{rot1}{rgb}{0.000, 0.000, 0.000}
\newcommand{\rot}{\textcolor{rot1}}

\begin{document}
\renewcommand{\theequation}{\arabic{section}.\arabic{equation}}
\begin{titlepage}
\title{\bf \tcr{Uniqueness to some inverse source problems for the wave equation in unbounded domains}}

\author{Guanghui Hu\thanks{Beijing Computational Science Research Center, Beijing 100193, China ({\sf hu@csrc.ac.cn}).} \qquad
Yavar Kian\thanks{Aix Marseille Univ, Universit\'e de Toulon, CNRS, CPT, Marseille,
France. ({\sf yavar.kian@univ-amu.fr}).}\quad \quad
 Yue Zhao\thanks{ School of Mathematics and Statistics, Central China Normal University, Wuhan 430079, China ({\sf zhaoyueccnu@163.com}).}
}
\date{}
\end{titlepage}
\maketitle

%\vspace{.2in}

\begin{abstract}
This paper is concerned with inverse acoustic source  problems in an unbounded domain with dynamical boundary surface data of Dirichlet kind. \tcr{The measurement data are taken at a surface far away from the source support.
We prove uniqueness in recovering source terms of the form $f(x)g(t)$ and $f(x_1,x_2,t) h(x_3)$, where $g(t)$ and $h(x_3)$ are given and $x=(x_1, x_2, x_3)$ is the spatial variable in three dimensions. Without these a priori information, we prove that
the boundary data of a family of solutions can be used to recover general source terms depending on both time and spatial variables. For moving point sources radiating periodic signals,
 the data recorded at four receivers are prove sufficient to uniquely recover the orbit function. Simultaneous determination of embedded obstacles and source terms was verified in an inhomogeneous background medium \tcr{using the observation data of infinite time period}. Our approach depends heavily on the Laplace transform.}

\vspace{.2in} {\bf Keywords:} Inverse source problems, Laplace transform, moving point source, uniqueness.

\end{abstract}

\section{Introduction}
Inverse source problems have significant applications in many scientific areas such as antenna synthesis and design, biomedical engineering, medical imaging and optical tomography. For a mathematical overview of various inverse source problems we refer to \cite{Isakov89} by Isakov where uniqueness and stability are discussed. An application in the fields of inverse diffraction and near-field holography was presented in \cite[Chapter 2.2.5]{GC}.

The approaches of applying Carleman estimate \cite{K1992} and unique continuation \cite{Tataru} for hyperbolic equations have been widely used in the literature, giving rise to uniqueness and stability results for inverse coefficient and inverse source problems with the dynamical data over a finite time; we refer to  \cite{ACY-EJAM04, CM2006,IY-01-IP, JLY2017, Ya1995, Ya1999} for an incomplete list. \tcr{Recently,
an inverse source problem for doubly hyperbolic equations arising from the nucleation rate reconstruction in the three-dimensional time cone model was analyzed in \cite{LJY-15-SIAM}. A Lipschitz stability result was proved for recovering the spatial component of the source term using interior data and an iterative thresholding algorithm (see also \cite{JLY-2017} with the final observation data) was tested.
However, most of the above mentioned works dealt with recovery of time independent source terms. We refer to \cite{DOT,RS,BHKY,HK} where specific time-dependent source terms for hyperbolic equations were considered and to \cite{kian2016} for the recovery of some class of space-time-dependent source terms in the parabolic equation on a wave guide. }
In the time-harmonic case, inverse source problems with multi-frequency data have been extensively investigated. The increasing stability analysis in recovering spatial-dependent source terms has been carried out from both theoretical and numerical points of view (see e.g., \cite{BLLT2015, BLRX, BLT-JDE, BLZ2017, CIL2016, EV2009, LY, LZ-AA}).

In the time domain, it is very natural to transform the wave scattering problem governed by hyperbolic equations into elliptic inverse problems in the Fourier or Laplace domain with \emph{multi-frequency} data; see e.g. \cite{Isakov08} for determining sound-hard and impedance obstacles in a homogeneous background medium.
In \cite{CIL2016}, the time-domain analysis helps for deriving an increasing stability to time-harmonic inverse source problems via Fourier transform.
The same idea was used in \cite{BHKY,HLLZ,HKLZ} for recovering  spatial-dependent sources as well as moving source profiles and orbits in elastodynamics and electromagnetism. The aim of this paper is to analyze the acoustic counterpart with new uniqueness results. \textcolor{rot}{Specially, this paper concerns the following four inverse problems with a single boundary surface data}:

\begin{enumerate}

\item \textcolor{rot}{Simultaneous determination of sound-soft obstacles and separable source terms in an inhomogeneous medium (Subsection \ref{sub1}).}

\item \textcolor{rot}{Simultaneous determination of sound-soft obstacles and general time-dependent source terms from a family of solutions (Subsection \ref{sub2}).}

\item \textcolor{rot}{Inverse moving point source problems from the data of four receivers (Subsection \ref{sub3.1}) .}

\item \textcolor{rot}{Determination of source terms which are independent of one spatial variable (Subsection \ref{sub3.2}) .}

\end{enumerate}
\tcr{The Laplace (Fourier) transform will be used to handle the above inverse problems 1,  2 and 4.}
We highlight the novelty of this paper as follows.
First, we verify the unique determination of both embedded obstacles and spatial-dependent source terms in an inhomogeneous medium. %This is motivated by the uniqueness proof of \cite{HK} for recovering elastic source terms in the exterior of a given cavity embedded in an inhomogeneous elastic medium.
Although the acoustically sound-soft obstacles are considered within this paper, the proof carries over to other reflecting boundary conditions for impenetrable scatterers in acoustics and elastodynamics (see Remark \ref{rema1}). Second, the data of a family of solutions are used to recover a general source which depends on both time and space variables; Thirdly, the data of a finite number of receivers are proven sufficient to determine the orbit of a moving point source which radiates periodic temporal signals. This differs from inverse moving source problems of \cite{HKLZ}, where compactly supported temporal functions were considered and Huygens' principle was applied.
Our uniqueness proof seems new and leads straightforwardly to a numerical algorithm. Finally, the argument for recovering source terms independent of one spatial variable has simplified the corresponding proof in linear elasticity contained in \cite{HK}. \tcr{Note that, although the measurement data are taken on a spherical surface, our results carry over to other non-spherical surfaces straightforwardly. In particular,  Theorem \ref{th1}, \ref{th3} and \ref{th4} remain valid if the data are observed on any subset of a closed analytical surface with positive Lebesgue measure.}

The remaining part of this \tcr{paper} is divided into three sections. In the subsequent Section \ref{ip1}, we consider simultaneous determination of sound-soft obstacles and source terms via \tcr{the} Laplace transform. Section \ref{ip2}
 is devoted to the unique determination of time-dependent source terms  in a homogeneous background medium, including inverse moving source problems.
 Some remarks and open questions will be concluded in Section \ref{conclude}.

\section{Simultaneous determination of sound-soft obstacles and source terms}\label{ip1}

Consider the time-dependent acoustic wave propagation in an inhomogeneous background medium with an acoustic source outside a sound-soft obstacle modelled by (see \textcolor{rot}{Figure 1})
\begin{align}\label{eqn}
\frac{1}{c^2(x)}\textcolor{rot}{\partial^2_t}u(x,t) - \triangle u(x,t) = F(x,t), \quad x\in\mathbb R^3\backslash \bar{D},\quad t>0,
\end{align}
where $c(x)$ is the wave speed, $u(x,t)$ denotes the wave field, $D\subset \mathbb R^3$ represents the region of the sound-soft obstacle and $F(x,t)$ is the acoustic source \tcr{term}. Together with the above governing equation, we impose the homogeneous initial conditions
\begin{align}\label{ic}
u(x,0) = 0, \quad \partial_t u(x,0) = 0, \quad x \in \mathbb R^3\backslash \bar{D},
\end{align}
and the Dirichlet boundary condition on $\partial D$:
\begin{align}\label{bdc}
u(x,t) = 0, \quad (x,t)\in\partial D \times \mathbb R^+.
\end{align}

\begin{figure}\label{fig}
\center
\includegraphics[width=0.5\textwidth]{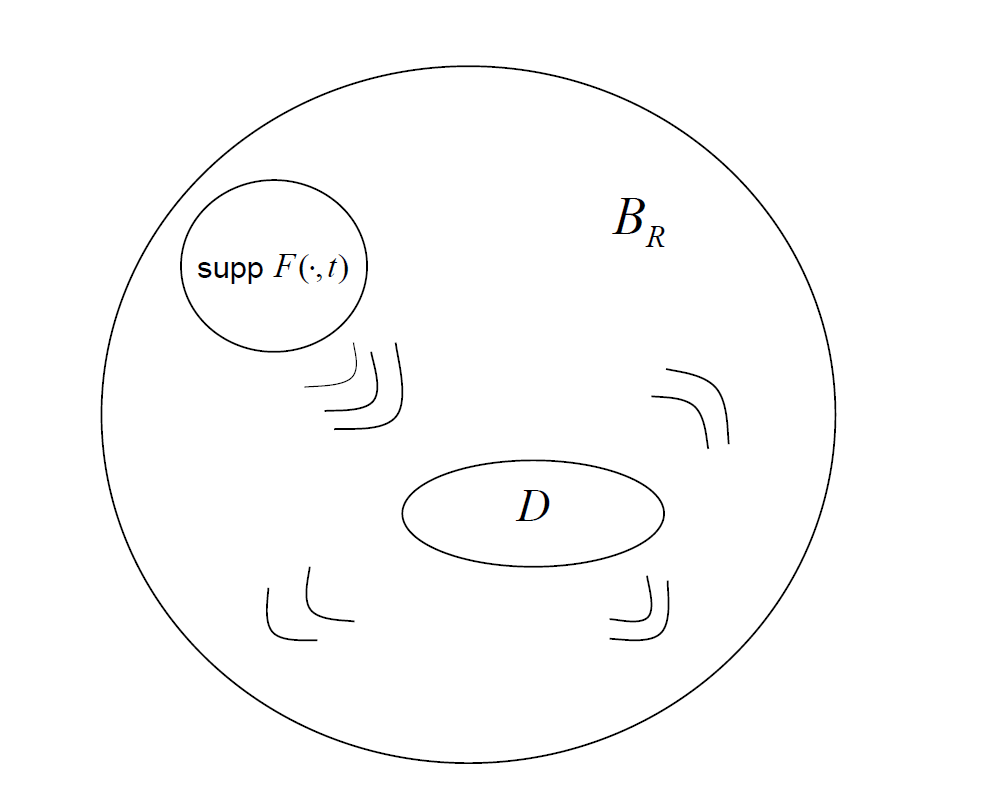}
\caption{Radiation of a source in the exterior of a sound-soft obstacle $D$ \tcg{in two dimensions}. The inverse problem is to determine both the source term \textcolor{rot}{$F = F(x, t)$} and the obstacle \textcolor{rot}{$D$} from the displacement data measured on $\Gamma_R:=\{x\in\mathbb R^3: |x|=R\}\tcr{=\partial B_R}.$}
\end{figure}

Throughout this paper we assume that \tcr{$D\subset B_R$ and that the source term $F(x,t)$ is compactly supported in $(B_R\backslash\overline{D}) \times (0,T_0)$}. Here $B_R:=\{x\in\mathbb R^3: |x|<R\}$ and $R>0$, $T_0>0$ are constants. We denote the boundary of $B_R$ by $\Gamma_R:=\{x\in\mathbb R^3: |x|=R\}$. It is also supposed that \textcolor{blue}{$c\in L^\infty(\mathbb R^3)$ satisfies
\begin{equation}\label{cc}\exists c_0>0,\quad c(x)\geq c_0\end{equation}}
 and $\text{supp}(1-c)\subset B_R$, which means that the \tcr{acoustic} medium outside $B_R$ is homogeneous. We also assume that
 \tcg{$\overline{\text{supp}(F(\cdot, t))} \cap \overline{D} = \emptyset$}
 \tcg{for all $t>0$}
   and that $D$ is a $\mathcal{C}^3$-smooth domain with the connected exterior $\mathbb R^3 \backslash \bar{D}$. Suppose that $F(x,t)\in L^2(0,T_0; L^2(B_R\backslash\overline{D}))$.
Then, the problem \eqref{eqn}-\eqref{bdc} admits a unique solution
\[
u \in \mathcal{C}^1([0, +\infty); L^2(\mathbb R^3 \backslash D))\cap \mathcal{C}([0, +\infty); H^1(\mathbb R^3\backslash D)).
\]
The proof of this result can be carried out using the elliptic regularity properties of the Laplace operator (see \cite{HK, HW, LM1, LM2, mclean}).

The goal of this section is to recover both the source term $F(x,t)$ and the embedded obstacle $D$ from the boundary surface data $\{u(x,t):~ |x|=R, ~t>0\}$ over an infinite time period. It is important to note that  uniqueness in recovering time-dependent source terms is not true in general. A non-uniqueness example can be easily constructed in the absence of the obstacle $D$ (that is, $D =\emptyset$). In fact,
let $\chi\in \mathcal{C}_0^{\infty}(B_R \times (0, T_0))\neq 0$ such that the function
\[
F(x,t):= \frac{1}{c^2}\textcolor{rot}{\partial^2_t}\chi - \triangle \chi,\quad (x,t)\in \R^3\times\R_+
\] does not vanish identically.
Consider the inhomogeneous source problem
\begin{align}\label{ce}
\begin{cases}
\frac{1}{c(x)^2}\textcolor{rot}{\partial^2_t}u(x,t) - \triangle u(x,t) = F(x,t), \quad x \in \mathbb R^3 \times (0, +\infty),\\
u(x,0) = \partial_t u(x,0) = 0, \quad x \in \mathbb R^3.
\end{cases}
\end{align}
Clearly, from the uniqueness
of solutions of \eqref{ce} we conclude that $u = \chi$ is the unique solution. However, we have
\[
u(x, t) = 0, \quad |x| = R,\; t \in (0, +\infty),
\]
due to the fact that $\text{supp}(\chi)\subset B_R \times (0, T_0)$. This means that $F \neq 0$ is a non-radiation source and thus the surface
  data $\{u(x,t):|x|=R, ~t>0\}$ usually do not allow the unique recovery of general source terms $F(x,t)$ satisfying $\text{supp}(F)\subset B_R \times (0, T_0)$. It implies that
there is no hope \tcr{to prove} uniqueness with a single measurement data.
Facing this obstruction, we need to either know a certain \emph{a prior} information of the source (see Subsections \ref{sub1} and \ref{sub3.2}) or make use of extra data (see subsection \ref{sub2}) for \tcr{recovering} both time- and spatial-dependent source terms.

\subsection{Spatial-dependent source terms in an inhomogeneous background medium}\label{sub1}

In this section we consider source terms of the form
\begin{align}\label{fg}
F(x,t)=f(x)\;g(t), \quad x \in \mathbb R^3 \backslash \bar{D}, ~t \in (0, \infty),
\end{align}
where  $f\in L^2(B_R\backslash \overline{D})$ is the spatial-dependent source term to be determined and $g\in L^2(0,T_0)$ is a given temporal function. \tcb{ We fix also $U$ an open and connected set of $\mathbb R^3$ such that $\overline{U}\subset B_R$.}
%$g$ is a real valued function. Throughout the paper, it is supposed that $c\in \mathcal{C}^2(\mathbb R^3)$.  Our first inverse problems in the exterior of the obstacle can be stated as follows.
%\textbf{Inverse Problem 1 (IP1)}: Assume the source term to take the form $F(x,t)=f(x)g(t)$, where $f\in L^2(\mathbb R^3\backslash D), g\in L^2(0,T)$ and $g$ is known in advance. Determine simultaneously the spatially dependent function $f$ and the obstacle $D$ from the radiated field $u(x,t)$ measured on the surface $\Gamma_R\times \mathbb R^+$.

Below we give a confirmative answer to the uniqueness issue of our inverse problem under proper assumptions on $supp(f_j)$ and $D_j$.
\tcb{\begin{theo}\label{th1}  Let $g\in L^2(0,T_0)$ and let $c\in L^\infty(\mathbb R)$ be such that $(1-c)$ is  supported in $B_R$ and \eqref{cc} is fulfilled. For $j=1,2$, let $D_j$ be an obstacle contained into $U$ and $f_j\in L^2(B_R\backslash \overline{D_j})$ satisfy $\tcg{\overline{supp(f_j)}\subset B_R\backslash\overline{U}}$ and $U\setminus\overline{D_j}$ is connected. Here we assume that $f_1,f_2$ are \tcg{non-uniformly} vanishing.
Denote by $G$ the connected component of $U\backslash\overline{D_1\cup D_2}$ which can be connected to $\R^3\backslash \overline{B_R}$.
 We assume that there exists $\mathcal O$ an open and connected subset of $\mathbb R^3$ such that
\begin{equation}\label{th1a} \mathcal O\cap(\mathbb R^3\backslash B_R)\neq\emptyset,\quad \mathcal O\cap G\neq\emptyset,\quad \mathcal O\cap \textrm{supp}(f_1-f_2)=\emptyset.\end{equation}
Then, for $u_j$ solving \eqref{eqn}-\eqref{bdc} with $F(x,t)=f_j(x)g(t)$ and $D=D_j$, the condition
\begin{equation}\label{th1b} u_1(x,t)=u_2(x,t),\quad x\in \Gamma_R, \, t>0,\end{equation}
implies $D_1=D_2$ and $f_1=f_2$.
\end{theo}}

If $f$ is known and $c(x)\equiv 1$, the unique determination of the sound-soft obstacle $D$ can be proved with the dynamical data over a finite time, following Isakov's idea of using the sharp unique continuation for hyperbolic equations with analytic coefficients; see \cite[Theorem 5.1]{Isakov}.
If the obstacle $D$ is absent and the background medium is homogeneous, it was shown in \cite{BHKY,HLLZ} via Huygens' principle and Fourier transform that the boundary surface data can be used to uniquely determine $f$ in both elastodynamics and electromagnetism. Below we shall prove \tcr{uniqueness in determining both} $D$ and $f$ in an inhomogeneous medium. For this purpose, we need to apply the Laplace transform in place of the Fourier transform, because the strong Huygens' principle is no longer valid. %and we cannot even expect integrable local energy decay (see e.g., \cite{kawa} in linear elasticity).

\textbf{Proof of Theorem \ref{th1}}.
%Assume there are two source functions $f_1$, $f_2$ and two sound-soft obstacles $D_1$, $D_2$ with the corresponding surface data satisfying
%\begin{align}\label{data}
%u_1(x,t)=u_2(x,t) \quad \text{on} ~ \Gamma_R \times \R^+,
%\end{align}
Obviously, $u_1$ and $u_2$ are solutions to
\begin{align}\label{1}
\begin{cases}
\frac{1}{c^2(x)}\textcolor{rot}{\partial^2_t}u_j(x,t) -  \triangle u_j(x,t) = f_j(x)g(t), \quad &(x,t) \in \mathbb R^3\backslash  \bar{D}_j\times \mathbb R^+,\\
u_j(x,0) = \partial_t u_j(x,0) = 0 , \quad & x \in \mathbb R^3\backslash \bar{D}_j,\\
u_j(x,t) = 0, \quad & (x,t) \in \partial D_j \times \mathbb R^+,
\end{cases}
\end{align}
for $j=1,2$.
By using standard argument for deriving energy estimates, we can prove that $u_j(x,t)$ $(j=1,2)$ has a long time behavior which is at most of polynomial type (see e.g., \cite[Proposition 9]{HK}). This allows us to define the Laplace transform of $u:=u_1-u_2$ with respect to the time variable as following:
\begin{align}\label{laplace}
\hat{u}(x,s):= \int_{\mathbb R}u(x,t)\textcolor{rot}{{\rm e}}^{-st}{\rm d}t, \quad s>0, \quad x\in B_R.
\end{align}
%\begin{align}\label{2}
%\begin{cases}
%\frac{1}{c^2(x)}\textcolor{rot}{\partial^2_t}u_2(x,t) -  \triangle u_2(x,t) = f_1(x)g(t), \quad &(x,t) \in \mathbb R^3\backslash  \bar{D}_2\times \mathbb R^+,\\
%u_2(x,0) = \partial_t u_2(x,0) = 0 , \quad & x \in \mathbb R^3\backslash \bar{D}_2,\\
%u_2(x,t) = 0, \quad & (x,t) \in \partial D_2 \times \mathbb R^+,
%\end{cases}
%\end{align}
%respectively.
Denote by $\tilde{D}$ the unbounded component of $\R^3\backslash\overline{D_1\cup D_2}$ and set $f:=f_1-f_2$. It then follows that
\begin{align*}
\begin{cases}
\frac{1}{c^2(x)}\textcolor{rot}{\partial^2_t}u(x,t) -  \triangle u(x,t) = f(x)g(t), \quad &(x,t) \in   \tilde{D} \times \mathbb R^+,\\
u(x,0) = \partial_t u(x,0) = 0 , \quad & x \in \tilde{D},\\
u(x,t) = 0, \quad & (x,t) \in \Gamma_R \times \mathbb R^+.
\end{cases}
\end{align*}
For notational convenience
we set $\Omega_1:= B_R \backslash \bar{D}_1$, $\Omega_2:= B_R \backslash \bar{D}_2$ and $\Omega = B_R\cap\tilde{D}.$

Since $\widehat{\textcolor{rot}{\partial^2_t}u}(x,s) = s^2\hat{u}(x,s)$ for all $s>0$ and the background wave speed $c(x)$ is known, the function $x\mapsto\hat{u}(x,s)$ solves
\begin{align}\label{bvp}
\begin{cases}
\triangle \hat{u}(x,s) - \frac{s^2}{c^2(x)}\hat{u}(x,s) = f(x)\;\hat{g}(s) \quad & \text{in} \quad \tilde{D},\\
\hat{u}(x,s)=0 \quad & \text{on} \quad \partial \tilde{D}.
\end{cases}
\end{align}
\tcb{Moreover, the uniqueness of solutions to the wave equation in the unbounded domain $|x|>R$ with the homogeneous Dirichlet boundary condition on $\Gamma_R \times (0, \infty)$, which can be justified via standard energy estimate (see e.g. \cite{HLLZ} for a proof in electromagnetism), implies that $u(x,t)=0$ for $(x,t)\in(\mathbb R^3\backslash B_R)\times (0, \infty)$.  By Laplace transform,
this gives the relation $\hat{u}(x,s)=0$ for $(x,s)\in(\mathbb R^3\backslash B_R) \times (0, \infty)$.} \tcb{In view of \eqref{th1a}, fixing $\mathcal O_R=\mathcal O\cap (\mathbb R^3\backslash B_R)$, we deduce that, for all $s>0$, the restriction on $\mathcal O$ of $\hat{u}(\cdot,s)$ solves
\begin{align}
\begin{cases}
\triangle \hat{u}(x,s) - \frac{s^2}{c^2(x)}\hat{u}(x,s) = 0 \quad & x\in \mathcal O,\\
 \hat{u}(x,s)=0 \quad &x\in \mathcal O_R.
\end{cases}
\end{align}
Applying unique continuation  results for elliptic equations (e.g. \cite[Theorem 1.1]{GL} and \cite[Theorem 1]{SS}), we deduce that \[
\hat{u}(x,s)=0 \quad \text{for all} ~ x\in\mathcal O.\]
In particular, we have
\[\hat{u}(x,s)=0 \quad \text{for all} ~ x\in \mathcal O\cap \tcg{G}\]
and we deduce that
\begin{align}
\begin{cases}
\triangle \hat{u}(x,s) - \frac{s^2}{c^2(x)}\hat{u}(x,s) = 0 \quad & x\in \tcg{G},\\
\hat{u}(x,s)=0 \quad &x\in \mathcal O\cap \tcg{G}.
\end{cases}
\end{align}
Applying again unique continuation results for elliptic equations and the fact that
\tcg{$O\cap (\R^3\backslash B_R)\neq\emptyset$, we deduce that
%$$U\backslash\overline{D_1\cup D_2}=(U\backslash\overline{D_1})\cap (U\backslash\overline{D_1})$$
%is connected open set, we deduce that
 \begin{align}\label{uc}
\hat{u}(x,s)=0 \quad \text{for all} ~ x\in\overline{G}.\end{align}}
}

We first prove $D_1=D_2$. Assuming on the contrary that $D_1\neq D_2$, we shall derive a contraction as follows.
Without loss of generality we may assume $D^* := (U\backslash G) \backslash \overline{D}_1 \neq \emptyset$. \rot{Then, using the fact that $\partial G\subset [\partial U\cup \partial (D_1\cup D_2)]$ it holds that
$$\partial D^*\subset [\partial D_1\cup (\partial G\backslash \partial U)]\subset \partial D_1\cup (\partial G\cap \partial D_2).$$}
On the other hand, from \eqref{uc} we deduce that $$\hat{u}_1(x,s)=\hat{u}_2(x,s)=0,\quad \rot{x\in\partial G\cap \partial D_2},\ s>0.$$ Therefore, combining this with the fact that
$$\hat{u}_1(x,s)=0,\quad x\in\partial D_1,\ s>0,$$
we deduce that $\hat{u}_1$ solves the boundary value problem
\begin{align*}
\begin{cases}
\triangle \hat{u}_1(x,s) - \frac{s^2}{c^2(x)}\hat{u}_1(x,s) = 0 \quad & \text{in} ~ D^*,\\
\hat{u}_1(x,s) = 0 \quad &\text{on} ~ \partial D^*\,\tcr{,}
\end{cases}
\end{align*}
 On the other hand, for all $s>0$, $0$ is not in the spectrum of the operator $-\triangle  + \frac{s^2}{c^2(x)}$ with Dirichlet boundary condition on $D^*$ which is contained into
$\left[\frac{s^2}{\| c\|_{L^\infty( D^*)}},+\infty\right)$. Therefore, we have $\hat{u}_1(\cdot, s)\equiv 0$ in $D^*$.
 Applying unique continuation,
 we get $\hat{u}_1(\cdot,s) = 0$ in $U\backslash\overline{D_1}$ for each $s>0$. In the same way, applying \eqref{th1a} we deduce that $\hat{u}_1(\cdot,s) = 0$ in $\mathcal O$ and then that $\hat{u}_1(\cdot,s) = 0$ on $\mathbb R^{3}\setminus \overline{B_R}$. Here we use the fact that $(\mathbb R^{3}\setminus \overline{B_R})\cap\mathcal O\neq\emptyset.$

\tcb{For $R_1>R$, we fix  $\tilde{\Omega}_1=B_{R_1}\backslash\overline{D_1}$ and $\Omega^*_1=B_{R_1}\backslash\overline{B_R}$. Then we have
 \ben
 \triangle \tcr{\hat{u}_1}(x,s) - \frac{s^2}{c^2(x)}\tcr{\hat{u}_1}(x,s) = \tcr{f_1}(x)\;\hat{g}(s) \quad & \text{in} \quad \tilde{\Omega}_1,
 \enn and
\begin{align*}
\hat{u}_1(x,s) = 0,\quad x\in\Omega^*_1
\end{align*}
for every \tcr{$s>0$}. Denote by $\langle \cdot, \cdot\rangle$ the inner product in $L^2(\tilde{\Omega}_1;c^{-2}dx)$, i.e.,
\[
\langle u, v\rangle := \int_{\tilde{\Omega}_1}c^{-2}(x)u(x)\bar{v}(x){\rm d}x, \quad u, v\in L^2(\tilde{\Omega}_1).
\]
Denote by $\{\gamma_l, \phi_{l, k}(x)\}_{l\in \N^+, k\leq m_l}$ the eigenvalues and an associated orthonormal basis of eigenfunctions for the operator $-c^{2}(x)\triangle$ over $\tilde{\Omega}_1$ with the Dirichlet boundary condition acting on $L^2(\tilde{\Omega}_1;c^{-2}dx)$. Here the eigenvalues satisfy the relation $0<\gamma_1<\gamma_2<\cdots<\gamma_l<\cdots$ and $\{\phi_{l, k}\}_{k=1}^{m_l}$ denotes the eigenspace associated with $\gamma_l$.
In $\tilde{\Omega}_1$ we can represent the functions $c^2(x) f_1(x)$ and
 $\hat{u}_1(x,s)$ as
\be\nonumber
c^2(x) f_1(x)&=&\sum_{ l \in \mathbb N^+} \sum_{k=1}^{m_l}\langle c^{2}f_1, \phi_{ l, k}\rangle\phi_{ l, k}(x),\quad m_l\in \N^+, \\ \label{repre}
\hat{u}_1(x,s) &=& \hat{g}(s) \sum_{ l \in \mathbb N^+}\frac{\sum_{k=1}^{m_l}\langle c^{2}f_1, \phi_{ l, k}\rangle\phi_{ l, k}(x)}{s^2 + \gamma_{l}}, \quad s>0.
\en
}
Note that the convergence of the series (\ref{repre}) can be understood in $L^2(\tilde{\Omega}_1;c^{-2}dx)$. Since $g \in L^2(\mathbb R^+)$ is supported in $[0,T_0]$ and does not vanish identically, there exists an interval $I \subset (0,+\infty)$ such that $|\hat{g}(s)|>0$ for \tcr{all} $s \in I$.
% Moreover, the sequence
%\begin{align*}
%\hat{u}_{1,N}(x,s) = \hat{g}(s) \sum_{ l=1}^N\frac{\sum_{k=1}^{m_l}\langle c^{2}f_1, \phi_{ l, k}\rangle\phi_{ l, k}(x)}{s^2 + \gamma_{l}}
%\end{align*}
%converges to $\hat{u}_1(x,s)$ as $N\rightarrow \infty$ in the sense of $L^2_{1/c}(\Omega_1)$. Since $L^2_{1/c}(\Omega_1)$ is embedded continuously into $L^2(\Omega^*)$, we deduce that $\hat{u}_{1,N}(x,s)|_{x\in \Omega^*_1}$ also converges to $\hat{u}_{1}(x,s)|_{x\in \Omega^*_1}$ in the sense of $L^2(\Omega^*_1)$.
Recalling that $\hat{u}_1(x,s) = 0$ in \tcr{$\Omega_1^*$}, we have for all $s\in I$ that
\[
 \sum_{ l \in \mathbb N^+}\frac{\sum_{k=1}^{m_l}\langle c^{2}f_1, \phi_{ l, k}\rangle\phi_{ l, k}(x)}{s^2 + \gamma_{l}} = 0 \quad \text{for}~ \text{a.e.} ~x\in \Omega^*_1.
\]
On the other hand, the function
\[
G(x,z):\quad z\rightarrow \sum_{ l \in \mathbb N^+}\frac{\sum_{k=1}^{m_l}\langle c^{2}f_1, \phi_{ l, k}\rangle\phi_{ l, k}(x)}{z + \gamma_{l}}, \quad z\in\mathbb C \backslash \{-\gamma_l: l\in \mathbb N^+\}
\]
can be regarded as a holomorphic function in the variable $z$ taking values in $L^2(\Omega^*_1)$. Hence, by unique continuation for holomorphic functions we deduce that the condition
\[
G(x,s^2) = \tcr{\hat{u}_1}(x,s)|_{x\in \Omega^*_1} = 0 \quad \text{for all}\quad s\in I
\]
implies that
\[
G(x,z) = 0 \quad \text{for all}\quad z\in\mathbb C \backslash \{-\gamma_l: l\in \mathbb N^+\}.
\]
It follow that
\begin{align}\label{2.8}
(z + \gamma_j) G(x,z) = 0, \quad z\in\mathbb C \backslash \{-\gamma_l: l\in \mathbb N^+\}, \quad j\in \mathbb N^+.
\end{align}
Therefore, letting $z\rightarrow -\gamma_j$ in \eqref{2.8} yields
\[
\phi_j(x):= \sum_{k=1}^{m_j}\langle c^{2}f_1, \phi_{ j, k}\rangle\phi_{j, k}(x) = 0\quad \text{for} \quad x\in \Omega^*_1.
\]
On the other hand, we deduce that $\phi_j$ satisfies the elliptic equation
\[
\triangle \phi_j(x) + \frac{\gamma_j}{c^2(x)}\phi_j(x) = 0, \quad x\in \Omega^*_1,
\]
since $\phi_{ l, k}$ are eigenfunctions.
Applying the unique continuation of the Helmholtz equation gives
\[
\sum_{k=1}^{m_j}\langle c^{2}f_1, \phi_{j, k}\rangle\phi_{j, k}(x) = 0\quad \text{for} \quad x\in \tilde{\Omega}_1,
\]
leading to the relations
\[
\langle c^{2}f_1, \phi_{j, k}\rangle = 0, \quad k=1,2, \cdot\cdot\cdot, m_j.
\]
Finally, by the arbitrariness of $j \in \mathbb N^+$ and the fact that $\text{supp}(f_1)\subset \tilde{\Omega}_1$, we obtain
\[
f_1 = c^{-2}(c^2f_1)\equiv0 \quad \text{for} ~x\in \tilde{\Omega}_1,
\]
which is a contradiction to $f_1 \neq 0$ in $\tilde{\Omega}_1$. Thus, we obtain $D_1=D_2$.

It remains to prove the coincidence of the source $f_1=f_2$. We shall deduce $f=f_1-f_2\equiv0$ from the boundary value problem \eqref{bvp} in an open set $\tilde{\Omega}$ such that $\text{supp}(f) \subset \tilde{\Omega} \subset \subset \Omega$. \tcr{It is easy to prove that $\hat{u}(x,s)$ vanishes in $\tilde{\Omega}\backslash\overline{\text{supp}(f)}$.
Similarly to \eqref{repre}, we can represent $c^2(x) f(x), \hat{u}(x,s)$} in the form of \eqref{repre} in
$\tilde{\Omega}$.
Consequently, following similar arguments in the first step we can obtain $f=0$ in \tcr{$\tilde{\Omega}$} by making use of the vanishing of $u$ in
$\tcr{\tilde{\Omega} \backslash\overline{\text{supp}(f)}}$.
\qed

\begin{rema}\label{rema1}
\tcb{(i) Assuming that $c\in\mathcal C^1(\mathbb R^n)$, one can apply the local unique continuation results of \cite[Theorem 1]{Ta} in order to  derive  a global Holmgren uniqueness theorem similar to \cite[Theorem 3.16]{KKL} (see also \cite[Theorem A.1.]{KMO}). Combining this with the arguments used in \cite[Theorem 2]{HK} it is possible to prove Theorem \ref{th1} in a more straightforward way. However, for more general coefficients $c\in L^\infty(\mathbb R^n)$, it is not clear that \cite[Theorem 1]{Ta} holds true and we can not apply such arguments. In that sense, in contrast to \cite[Theorem 2]{HK}, the approach considered in Theorem \ref{th1} can be applied to equations with less regular coefficients.\\}

(ii)The result of Theorem \ref{th1} carries over to
other boundary conditions of the form
\ben
\partial_\nu u-a\partial_t u-b u=0\quad\mbox{on}\quad \partial D\times \R^+,
\enn where $a\geq 0$ and $b\geq 0$. The proof can be carried out by applying the Laplace transform with the variable $s=s_1+is_2\in \C^+$ such that $s_1, s_2>0$; we refer to \cite{Isakov08} by Isakov where
uniqueness results for recovering impenetrable obstacles were discussed.
 Note that although the gap domain $D^*$ between two obstacles might be cuspidal and non-lipschitzian, the regularity assumption of $\partial D$ ensures that $\hat{u}(\cdot,s)\in H^2(B_R\backslash\overline{D})$ and \tcr{the boundary $\partial D^*$ of the gap domain is piecewise smooth.}
Hence, the traces $\hat{u}(x,s)$ and $\partial_\nu\hat{u}(x,s)$ are well defined on $\tcr{\partial D^*}$. %The unique determination of $\partial D$ with a single surface data also follows from the sharp unique continuation for hyperbolic equations; we refer to \cite[Theorem 5.1]{Isakov}  for the proof for
 %sound-hard scatterers (with the Neumann boundary condition) in a homogeneous background medium.  }
 However, it remains unclear to us how to treat penetrable scatterers with transmission conditions on the interface.

(iii) The proof of Theorem \ref{th1} can be simplified if the background medium is homogeneous, i.e., $c(x) \equiv 1$ in $\R^3\backslash\overline{D}$. In fact, \tcr{in a homogeneous medium the uniqueness proof} can be reduced to verifying the vanishing of $f_1$ if
%Specifically, repeating the proof above, by elliptic analyticity, we have $\hat{u}_1(x,s) = 0$ in $\Omega_1 \backslash \text{supp}f_1$ for each $s>0$. Without loss of generality, there exists a ball $B$ such that $\text{supp}f \subset B, D\cap B = \phi$ and $\hat{u}_1(x,s) = \partial_{\nu}\hat{u}_1(x,s)(x,s) = 0$ for all $s>0$ and $x\in \partial B$ by the Dirichlet-to-Neumann map. In addition,  we have
\ben\label{bvp1}
\triangle \hat{u}_1(x,s) - s^2 \hat{u}_1(x,s) = \hat{g}(s)f_1(x), &&\quad x \in \tcr{\tilde{\Omega}_1},\\
\hat{u}_1(x,s) = \partial_{\nu}\hat{u}_1(x,s)(x,s) = 0, &&\quad x\in \partial \tcr{\tilde{\Omega}_1}
\enn for each $s>0$ and for some domain \tcr{$\tilde{\Omega}_1$} containing ${\rm supp}(f_1)$.
Multiplying both sides of \eqref{bvp1} by the test function $\varphi(x) = \textcolor{rot}{{\rm e}}^{s x \cdot d}$ with $d\in \R^3$, $|d|=1$ and integrating by parts over $B$ yield
\begin{align}\label{ftf}
\hat{g}(s)\int_{\tcr{\tilde{\Omega}_1}} f_1(x)\textcolor{rot}{{\rm e}}^{s x \cdot d}{\rm d}x = 0\quad\mbox{for all}\quad s\in \R.
\end{align}
Clearly,  $\hat{g}(z)$ and $\int_{\tcr{\tilde{\Omega}_1}} f_1(x)\textcolor{rot}{{\rm e}}^{z x \cdot d}{\rm d}x$ are both holomorphic functions with respect to the variable $z \in \mathbb C$. Using the assumption
 $g \neq 0$ it is easy to prove that $\int_{\tcr{\tilde{\Omega}_1}} f_1(x)\textcolor{rot}{{\rm e}}^{s x \cdot d}{\rm d}x = 0$ for all \tcr{$s\in \mathbb R$ and $|d|=1$}. \tcr{This implies that the Laplace transform of $f_1$ vanishes everywhere and hence $f_1\equiv0$.} %and is a contradiction. Hence we have $D_1 = D_2$. Repeating this process we can also prove that $f_1 = f_2$.
\end{rema}
Consider the acoustic wave equation with a homogeneous source term and inhomogeneous initial conditions $v_0$ and $v_1$:
\begin{align}\label{II}
\begin{cases}
\frac{1}{c^2(x)}\textcolor{rot}{\partial^2_t}u(x,t) - \triangle u(x,t) = 0, \quad &x=(x_1,x_2,x_3)\in\mathbb R^3\backslash \bar{D}, t>0,\\
u(x,0) = v_0(x), \partial_t u(x,0) = v_1(x),\quad &x \in \mathbb R^3\backslash D,\\
u(x,t) = 0, \quad &(x,t)\in \partial D \times \mathbb R^+.
\end{cases}
\end{align}
Applying the Laplace transform to $u$ and noting that $\widehat{\textcolor{rot}{\partial^2_t}u}(x,s) = s^2\hat{u}(x,s)-v_{1}-sv_{0}$ yield the boundary value problem
\begin{align*}
\begin{cases}
\triangle \hat{u}(x,s) - \frac{s^2}{c^2(x)}\hat{u}(x,s) = \frac{1}{c^2(x)}(sv_0(x) + v_1(x)) \quad & \text{in} ~ R^3\backslash \bar{D},\\
\hat{u}(x,s) = 0 \quad &\text{on} ~ \partial D\,.
\end{cases}
\end{align*}
Following similar arguments as those in the proof of Theorem \ref{th1}, we can determine simultaneously the obstacle $D$, the initial displacement \textcolor{rot}{$v_0$} and initial velocity \textcolor{rot}{$v_1$} from the radiated field $u$ measured on the surface $\Gamma_R \times \mathbb R^+$.
\tcb{\begin{coro}\label{coro2.3} Let  $c\in L^\infty(\mathbb R)$ be such that $(1-c)$ is  supported in $B_R$ and \eqref{cc} is fulfilled. For $j=1,2$, let $D_j$ be an obstacle contained into $U$, $v_{j,0}\in H^2(\R^3\backslash\overline{D})$ and  $v_{j,1}\in H^1(\R^3\backslash\overline{D})$  satisfy
\tcg{$\overline{supp(v_{j,0})}\cup \overline{supp(v_{j,1})}\subset B_R\backslash\overline{U}$}
 with $U\setminus\overline{D_j}$  connected. Here we assume that $v_{j,0}$, $v_{j,1}$, $j=1,2$, are \tcg{non-uniformly} vanishing. Assume also that there exists $\mathcal O$ an open and connected subset of $\mathbb R^3$ such that \eqref{th1a} is fulfilled \tcg{with the last relation replaced by
 \[
 \mathcal O\cap \textrm{supp}(v_{1,j}-v_{2,j})=\emptyset,\quad j=0,1.
  \] }
 Then, for $u_j$ solving \eqref{II} with $v_0=v_{j,0}$, $v_1=v_{j,1}$ and $D=D_j$, the condition
\begin{equation}\label{c1a} u_1(x,t)=u_2(x,t),\quad x\in \Gamma_R, \, t>0,\end{equation}
implies $D_1=D_2$, $v_{1,0}=v_{2,0}$ and $v_{1,1}=v_{2,1}$.
\end{coro}}

\begin{rema} \begin{itemize}
\item[(i)] \tcb{ Like Theorem \ref{th1},  for $c\in\mathcal C^1(\mathbb R^n)$  one can deduce and even improve Corollary \ref{coro2.3} by using an approach based on unique continuation properties with arguments borrowed from \cite[Theorem 1]{Ta}, \cite[Theorem 3.11]{KKL} and \cite[Theorem 2]{HK}.  However, since for $c\in L^\infty(\mathbb R^n)$, it is not clear that \cite[Theorem 1]{Ta} holds true,  we can not consider such approach. In that sense, in contrast to other similar results, Corollary \ref{coro2.3} can be applied to equations with less regular coefficients $c$.\\}
 \item[(ii)] \tcr{The results of Theorem \ref{th1} and Corollary \ref{coro2.3} hold true with a finite time observation data on $\Gamma_R\times(0,T)$ if $g(0)\neq 0$. In fact, by Duhamel's principle, we may represent $u_j$ to the equation (\ref{1}) as
     \be\label{vj}
     u_j(x,t)=\int_0^t g(t-s)\,v_j(x,s)\,ds,\qquad t\in(0,\infty),
     \en where $v_j$ solves the initial value problem of the homogeneous wave equation
 \ben
\begin{cases}
\frac{1}{c^2(x)}\textcolor{rot}{\partial^2_t}v_j(x,t) -  \triangle v_j(x,t) = 0, \quad &(x,t) \in \mathbb R^3\backslash  \bar{D}_j\times \mathbb R^+,\\
v_j(x,0) =0,\quad \partial_t v_j(x,0) = f_j(x) , \quad & x \in \mathbb R^3\backslash \bar{D}_j,\\
v_j(x,t) = 0, \quad & (x,t) \in \partial D_j \times \mathbb R^+.
\end{cases}
\enn
If $g(0)\neq 0$, differentiating (\ref{vj}) and then applying the Grownwall inequality could lead to the relation $v_1(x,t)=v_2(x,t)$ in $\{|x|>R\}\times (0,T)$, if $u_1(x,t)=u_2(x,t)$ on $\Gamma_R\times(0,T)$. Together with the unique continuation for the wave equation (\cite{EINT,ET}), this implies the coincidence of the initial velocities, i.e., $f_1=f_2$. The proof of $\partial D_1=\partial D_2$ can be proceeded analogously. In the case of the observation data over infinite time, one can also apply the Laplace transform to (\ref{vj}) to prove Theorem \ref{th1} and Corollary \ref{coro2.3}.}
\end{itemize}
\end{rema}
\subsection{General source terms in a family of controllable background media}\label{sub2}
As mentioned at the beginning of section \ref{ip1}, it is in general impossible to uniquely recover \tcr{a general source term of the form} $F(x,t)$, \tcr{due to the presence of time-dependent non-radiating sources}. This subsection is devoted to proving uniqueness with a family of solutions $u_\lambda(x, t)$ measured on $\Gamma_R\times \R^+$.

Consider the wave equations
\begin{align}\label{eqn2}
\begin{cases}
q_{\lambda}(x)\textcolor{rot}{\partial^2_t}u_{\lambda}(x,t) - \triangle u_{\lambda}(x,t) = F(x,t) \quad &\text{in} ~ \textcolor{rot}{\mathbb R^3 \backslash \bar{D}} \times \mathbb R^+,\\
u_{\lambda}(x,0) = \partial_t u_{\lambda}(x,0) = 0 \quad &\text{in}~\mathbb R^3,\\
u_{\lambda}(x,t)=0 &\text{on}~\partial D\times \R^+,
\end{cases}
\end{align}
where $q_{\lambda}(x)$ is the background medium function satisfying
\begin{align}\label{q}
q_{\lambda}(x)=
\begin{cases}
\lambda , \quad & x\in B_R,\\
1, \quad & x\in \mathbb R^3 \backslash \overline{B}_R.
\end{cases}
\end{align}
Our aim is to recover the compacted supported function $F$ from the data $\{u_{\lambda}(x,t): x\in \Gamma_R, ~ t>0, ~\lambda \in (a,b)\}$ for some \textcolor{rot}{$0<a<b$}. Physically, \tcr{such kind of the measurement data can be obtained by changing the background medium artificially and locally for the purpose of recovering a time-dependent source term which might be non-radiating for a fixed parameter}. Our uniqueness result below shows that any \tcr{compactly supported} acoustic source term cannot be a \tcr{non-radiating} source for a range of parameters $\lambda\in (a,b)$.

%\textbf{Inverse Problem 2} (IP2): Determine the obstacle $D$ and the general time-dependent source term $F(x,t)$ simultaneously from the wave field $u_{\lambda}(x,t)$ measured on the surface $\Gamma_R \times \mathbb R^+$ for $\lambda \in (a,b)$. Here $a$ and $b$ are two positive constants satisfying $a<b$.
%The above inverse problem describes the problem of recovering the unknown obstacle and source immersed in a known background medium. Furthermore, we are able to change the background medium gradually. For instance, assume that originally the background homogeneous medium is salt water where the density of the salt is a constant. Then by gradually adding more salt into the salt water we can change the density of the salt in the salt water, which results in the change of the medium function in an interval.
%Our uniqueness result for (IP2) can be stated as follows.
\tcb{\begin{theo}\label{th3}
 For $j=1,2$, let $D_j$ be an obstacle contained into $U$ and $F_j\in L^2((B_R\backslash \overline{D_j})\times \mathbb R_+)$ be supported on $(B_R\backslash \overline{D_j})\times [0,T]$, with $T>0$, satisfy
\begin{equation}\label{th3a}
\tcg{\overline{ \textrm{supp}(F_j(\cdot,t))}\subset B_R\backslash \overline{U}},\quad t\in(0,T)
\end{equation}
 and $U\setminus\overline{D_j}$ is connected. Here we assume that $F_1,F_2$ are \tcg{non-uniformly} vanishing. Assume also that there exists $\mathcal O$ an open and connected subset of $\mathbb R^3$ such that \eqref{th1a} is fulfilled \tcg{with the last relation replaced by
 \[
 \mathcal O\cap \textrm{supp}(F_1(\cdot, t)-F_2(\cdot, t))=\emptyset,\quad \mbox{for all}\quad t>0.
  \] }
 Then, for $u_{j,\lambda}$ solving \eqref{eqn2} with $\lambda \in (a,b)$, $F=F_j$ and $D=D_j$, the condition
\begin{equation}\label{th3b} u_{1,\lambda}(x,t)=u_{2,\lambda}(x,t),\quad x\in \Gamma_R, \, t>0,\ \lambda \in (a,b)\end{equation}
implies $D_1=D_2$ and $F_1=F_2$.  Here $a$ and $b$ are two positive constants satisfying $a<b$.
\end{theo}}
%\subsection{IP2}\label{ip2}
%
%According to \cite{BHKY}, there is an obstruction for the recovery of general time-dependent source terms $F$. In this subsection, we show that by using the set of data corresponding to different medium function $q_{\lambda}(x)$ in a finite interval the general source term $F(x,t)$ can be uniquely determined.
%
%\textbf{Proof of Theorem \ref{th3}}.
\begin{proof}
By our assumption, the function $u_{j,\lambda}$ ($j=1,2$) satisfies
\begin{align}\label{1.3}
\begin{cases}
q_{\lambda}(x)\textcolor{rot}{\partial^2_t}u_{j,\lambda}(x,t) -  \triangle u_{j,\lambda}(x,t) = F_j(x,t), \quad &(x,t) \in \mathbb R^3\backslash  \bar{D}_j\times \mathbb R^+,\\
u_{j,\lambda}(x,0) = \partial_t u_{j,\lambda}(x,0) = 0 , \quad & x \in \mathbb R^3\backslash \bar{D}_j,\\
u_{j,\lambda}(x,t) = 0, \quad & (x,t) \in \partial D_j \times \mathbb R^+.
\end{cases}
\end{align}

%and
%\begin{align}\label{2.3}
%\begin{cases}
%q_{\lambda}(x)\textcolor{rot}{\partial^2_t}u_{2,\lambda}(x,t) -  \triangle u_{2,\lambda}(x,t) = F_2(x,t), \quad &(x,t) \in \mathbb R^3\backslash  \bar{D}_2\times \mathbb R^+,\\
%u_{2,\lambda}(x,0) = \partial_t u_{2,\lambda}(x,0) = 0 , \quad & x \in \mathbb R^3\backslash \bar{D}_2,\\
%u_{2,\lambda}(x,t) = 0, \quad & (x,t) \in \partial D_2 \times \mathbb R^+,
%\end{cases}
%\end{align}
%respectively.
%Denote $D = D_1 \cup D_2$. Subtracting \eqref{1.3} by \eqref{2.3} and letting $u_{\lambda} = u_{1,\lambda} - u_{2,\lambda}$ and $F=F_1-F_2$ yield
%\begin{align*}
%\begin{cases}
%q_{\lambda}(x)\textcolor{rot}{\partial^2_t}u_{\lambda}(x,t) -  \triangle u_{\lambda}(x,t) = F(x,t), \quad &(x,t) \in \mathbb R^3\backslash  \bar{D} \times \mathbb R^+,\\
%u_{\lambda}(x,0) = \partial_t u_{\lambda}(x,0) = 0 , \quad & x \in \mathbb R^3\backslash \bar{D},\\
%u_{\lambda}(x,t) = 0, \quad & (x,t) \in \partial D \times \mathbb R^+.
%\end{cases}
%\end{align*}
%Denote $\Omega_1:= B_R \backslash \bar{D}_1$, $\Omega_2:= B_R \backslash \bar{D}_2$ and $\Omega = B_R\backslash \bar{D}.$ By similar arguments as in Theorem \ref{th1} we obtain
%\begin{align}\label{uc.3}
%\hat{u}_{\lambda}(x,s)=0 \quad \text{in} ~ \Omega \backslash (\text{supp}f_1 \cup \text{supp}f_2 \cup D).
%\end{align}
We first prove $D_1=D_2$. If $D_1\neq D_2$,
suppose without loss of generality that
\rot{$D^* := (U\backslash G) \backslash \overline{D}_1 \neq \emptyset$ where $G$ denotes the connected component of $U\backslash (D_1\cup D_2)$ which can be connected to $|x|>R$. }
As done in the proof of Theorem \ref{th1}, one can prove that
\begin{equation}\label{th3c}\hat{u}_{1,\lambda}(x,s)=0,\quad x\in \mathbb R^{3}\setminus \overline{B_R},\ s>0,\ \lambda\in(a,b).\end{equation}
\tcb{ For $R_1>R$, we fix  $\tilde{\Omega}_1=B_{R_1}\backslash\overline{D_1}$ and $\Omega^*_1=B_{R_1}\backslash\overline{B_R}$. Then, in a similar way to Theorem \ref{th1} we can prove that, for all $s>0$, we have
\[ -\triangle\hat{u}_{1,\lambda}(x,s) + \lambda s^2\hat{u}_{1,\lambda}(x,s) = \hat{F}_1(x,s) \quad  x\in\tilde{\Omega}_1\backslash\Omega^*_1,\ \lambda \in (a,b),\]}

\tcb{and
\begin{align*}
\hat{u}_{1,\lambda}(x,s) = 0,\quad x\in\Omega^*_1,\ \lambda \in (a,b).
\end{align*}
 Therefore, we get
\begin{align}\label{eqq}
\begin{cases}
 -\triangle\hat{u}_{1,\lambda}(x,s) + \lambda s^2\hat{u}_{1,\lambda}(x,s) = \hat{F}_1(x,s) \quad & \text{in} \quad \tilde{\Omega}_1,\\
\hat{u}_{1,\lambda}(x,s) = 0,\quad x\in\Omega^*_1.
\end{cases}
\end{align}
From now on we fix $s>0$.}
\tcb{Denote by $\langle \cdot, \cdot\rangle$ the inner product in $L^2(\tilde{\Omega}_1)$, i.e.,
\[
\langle u, v\rangle := \int_{\tilde{\Omega}_1}u(x)\bar{v}(x){\rm d}x, \quad u, v\in L^2(\tilde{\Omega}_1).
\]
Denote by $\{\gamma_{l,s}, \phi_{l, k,s}(x)\}_{l\in \N^+, k\leq m_l}$ the eigenvalues and an associated orthonormal basis of  eigenfunctions of the operator $-s^{-2}\triangle$ over $\tilde{\Omega}_1$ with the Dirichlet boundary condition acting on $L^2(\tilde{\Omega}_1)$. Here the eigenvalues satisfy the relation $0<\gamma_{1,s}<\gamma_{2,s}<\cdots<\gamma_{l,s}<\cdots$ and $\{\phi_{l, k,s}\}_{k=1}^{m_l}$ denotes the eigenspace associated with $\gamma_{l,s}$.}
\tcb{In $\tilde{\Omega}_1$ we can represent the functions $s^{-2} \hat{F}_1(x,s)$} \tcb{and
 $\hat{u}_1(x,s)$ as
$$
\hat{u}_{1,\lambda}(s,x) = \sum_{ l \in \mathbb N^+}\frac{\sum_{k=1}^{m_l}\langle s^{-2}\hat{F}_1(\cdot,s), \phi_{ l, k,s}\rangle\phi_{ l, k,s}(x)}{\lambda + \gamma_{l,s}}, \quad \lambda\in(a,b).
$$
Following, the proof Theorem \ref{th1}, combining this representation with the fact that
$$\hat{u}_{1,\lambda}(x,s) = 0,\quad x\in\Omega^*_1,\ \lambda\in(a,b),$$
we deduce that $\hat{F}_1(\cdot,s)=0$. This last identity holds true for any $s>0$ and the injectivity of the Laplace transform implies that $F_1\equiv0$ which is a contradiction with the condition imposed on $F_1$.} Hence we have $D_1 = D_2$. In a similar manner we can prove $F_1(x,t) = F_2(x,t)$.
\end{proof}
\tcg{\begin{rema} The condition \eqref{th1a} can always be fulfilled  if
 $B_{R}\backslash(\text{supp}(F_1(\cdot, t)) \cup  \text{supp}(F_2(\cdot,t)) $ is  connected uniformly for all $t>0$.  Under the additional assumption that
  $B_{R}\backslash(\text{supp}(F_j(\cdot, t)) $ ($j=1,2$) are both connected, the domain $\Omega_1^*$ in \eqref{eqq} can be chosen to be a neighboring area of
 $ \text{supp}(\hat{F}_1(\cdot, s))$ uniformly in all $s>0$. Then the vanishing of $F_1$ simply follows by multiplying $e^{s\sqrt{\lambda}x}$ on both sides of the equation in (\ref{eqq}) and then using integration by parts over $\Omega_1^*$. Note that the Cauchy data of $\hat{u}_{1,\lambda}$ vanish on $\partial \Omega_1^*$ in this case.
   \end{rema}
}
\begin{rema}

Consider the time-harmonic acoustic wave equation with a wave-number-dependent source term modelled by
\begin{align}\label{cmp}
\triangle u_{\lambda} + \kappa^2q_{\lambda}(x)u_{\lambda} = f(x,\kappa), \quad x \in \mathbb R^3,
\end{align}
where $\text{supp}\,f(\cdot, k )\subset B_{R}$ for each $k>0$ and $q_{\lambda}(x)$ is specified by \eqref{q}. Further, we suppose that $u_\lambda(x)$ fulfills the Sommerfeld radiation condition
\ben
r(\partial_r u_\lambda-ik u_\lambda)\rightarrow 0, \quad \mbox{as}\quad r=|x|\rightarrow\infty,
\enn
uniformly in all $\hat{x}=x/|x|$.
The proof of Theorem \ref{th3} implies that %by using the set of boundary observation data corresponding to different $q_{\lambda}(x)$ in a finite interval, we have the following uniqueness result of determining the general wave-number-dependent source term $f(x,\kappa)$. Here we give a simplified proof.
%\begin{theo}
the data $\{u_{\lambda}(x,\kappa): x\in \Gamma_R, ~\lambda \in (a,b), \kappa\in(\kappa_{min}, \kappa_{max})\}$ uniquely determine $f(x,\kappa)$ for all $x\in B_{R}$ and $\kappa>0$. Here,
$0<\kappa_{min}<\kappa_{max}$.
%\end{theo}
\end{rema}

%\begin{proof}
%It suffices to prove $f(x,\kappa) = 0$ if $u_{\lambda}(x,\kappa) = 0$ for all $\lambda \in (a,b)$, which also gives $\partial_{\nu}u_{\lambda}(x,\kappa) = 0$ for all $\lambda \in (a,b)$ by the Dirichlet-to-Neumann map.
%We introduce the test function $\varphi_{\lambda}(x) = \textcolor{rot}{{\rm e}}^{{\rm i}\kappa\sqrt{\lambda}x \cdot d}$ where $d \in \mathbb{S}^2$. Multiplying both sides of \eqref{cmp} by $\varphi_{\lambda}(x)$ and integrating by parts yield
%\[
%\int_{B_R}f(x,\kappa)\textcolor{rot}{{\rm e}}^{{\rm i}\kappa\sqrt{\lambda}x \cdot d}{\rm d}x = 0 \quad \text{for ~ all}~\lambda \in (a,b), ~ d \in \mathbb{S}^2.
%\]
%By analyticity of the Fourier transform of $f(x,\kappa)$ we have $f(x,\kappa) = 0.$

%\end{proof}

\section{Determination of other time-dependent source terms}\label{ip2}
This section is devoted to the unique determination of other two time-dependent source terms. For simplicity we suppose that the background medium is homogeneous and isotropic without embedded obstacles. In particular, we are interested in the \tcr{inverse problem of detecting of the track} of a moving point source.
\subsection{Moving point sources}\label{sub3.1}
Consider the acoustic wave propagation incited by a moving point source in a homogeneous medium modelled by
\begin{align}\label{eqn5}
\begin{cases}
\textcolor{rot}{\partial^2_t}u(x,t) - \triangle u(x,t) = \delta(x-a(t))\cos(\omega t), \quad &(x,t)\in \mathbb R^3\times \mathbb R^+\backslash\{(a(t),t): t\in \R^+\},\\
u(x,0) = \partial_t u(x,0) = 0, \quad & x\in \mathbb R^3,
\end{cases}
\end{align}
In \eqref{eqn5}, \tcr{the symbol} $\delta$ is the Dirac delta distribution in space, the function $a(t): [0, +\infty)\rightarrow \mathbb R^3$
models the orbit function of a moving source starting from the origin and $\cos(\omega t)$ is a cosine signal emitting from the moving source where $\omega>0$ denotes the frequency. Note that in this subsection the temporal function is not compactly supported in $\R^+$, differing from the other inverse problems of this paper. Physically, this means that the moving source radiates periodic signals continuously. It should be remarked that the relation between the orbit $a(t)$ and the signal $u(x, t)$ is non-linear and that the forward model cannot be understood in the time-harmonic sense.
We state
our inverse moving source problem as follows.

\textbf{Inverse Problem}: Determine the orbit function $\{a(t): t\in(0,T_0)\}$ from the radiated wave field $u$ detected at a finite number of receivers lying on the surface $\Gamma_R$ over the finite time period $(0,T)$ for some sufficiently large $\tcr{T>T_0>0}$.

%We refer to \cite{GF15} on inverse moving source problems using the time-reversal method and to \cite{PD89,PD91} for mathematical analysis on inverse
% moving obstacle problems. Numerical methods can be found in \cite{NIO12, wgll} to identify the orbit of a moving acoustic point source.
In the following uniqueness result we assume that $|a(t)|<R_1$ for some $0<R_1<R$ and all $t>0$, that is, the moving source does not enter into the exterior of $B_{R_1}$.

\begin{theo}\label{th6}
Assume $a(t)\in \mathcal{C}^2(0, +\infty)$, $|a^{\prime}(t)|<1$ and $a(0)=O$.
Let $x^{(j)}\in\Gamma_R$ ($j=1,\cdots, 4$) be four
receivers which do not lie on one plane. Then the orbit function over a finite interval of time $\{a(t): t\in (0,T_0)\}$ can be uniquely determined by the data $\{u(x^{(j)}, t)\}: j=1, \cdot\cdot\cdot, 4, t\in(0,T)$ for some $T>R+R_1+T_0$.
\end{theo}

%In this subsection, we connect the data measured on the sphere $\Gamma_R$ with the orbit $a(t)$ via an initial value problem for ordinary differential equations. Then we prove that the data measured at four discrete receivers on a sphere over a finite interval of time is sufficient to uniquely determine the orbit of the moving Dirac distribution over a finite interval of time.
\begin{proof}
Our proof relies on the distance function \textcolor{rot}{$t\mapsto|x-a(t)|$} between the receiver $x\in \Gamma_R$ and the source point $a(t)$ characterized by an ordinary differential equation with respect to $t>0$.

Firstly, we express the solution $u$ to the acoustic wave equation \eqref{eqn5} in terms of the Green's function as
\begin{align}\label{u}
u(x,t) &= \int_0^{\infty}\int_{\mathbb R^3} \frac{\delta(t-s-|x-y|)}{4\pi|x-y|}\delta(y-a(s))\cos(\omega s){\rm d}y{\rm d}s\notag\\
&= \int_0^{\infty} \frac{\delta(t-s-|x-a(s)|)}{4\pi|x-a(s)|}\cos(\omega s){\rm d}s.
\end{align}
Define $f(t) := t + |x-a(t)|\in \mathcal{C}^2(0, +\infty)$ for some fixed receiver $x\in\Gamma_R$.
Since $|a^{\prime}(t)|<1$, it is easy to see
\ben
f^{\prime}(t)=1+|x-a(t)|'=1-\frac{(x-a(t))\cdot a'(t)}{|x-a(t)|}>0,\quad t>0.
\enn
Note that $|x-a(t)|\neq 0$ for all $t>0$, due to the assumption $|a(t)|<R_1<|x|=R$. Hence, $f(t)> f(0)\tcr{=|x|}=R$ for all $t>0$.
From \eqref{u} we obtain
\begin{align}\label{uu}
u(x,f(t)) = \int_0^{\infty} \frac{\delta(f(t)-f(s))}{4\pi|x-a(s)|}\cos(\omega s){\rm d}s.
\end{align}
Change the variable by setting $\tau = f(s)$ in \eqref{uu}. Since $f$ is monotonically increasing in $\R^+$, its inverse $f^{-1}$ exists.
Consequently, we obtain
\be\nonumber
u(x,f(t)) &=& \int_{\tcr{R}}^\infty\left\{\frac{\delta(f(t)-\tau)}
{4\pi|x-a(s)|}\frac{\cos(\omega s)}{f'(s)}\Big|_{s=f^{-1}(\tau)}\right\}\;{\rm d}\tau\\ \nonumber
&=&\frac{1}
{4\pi|x-a(s)|}\frac{\cos(\omega s)}{f'(s)}\Big|_{s=t}\\ \label{ru}
&=&
\frac{\cos(\omega t)}{4\pi|x-a(t)|}\frac{1}{1+|x-a(t)|^{\prime}}.
\en
\tcr{Here we have used once again the fact that $f(t)>R$ for $t>0$.}
Denote the distance function between the receiver $x$ and the source position at the time point $t$ by $g(t) := |x-a(t)|=f(t)-t\in \mathcal{C}^2(0, +\infty)$. It follows from \eqref{ru} that $g(t)$ fulfills the ordinary differential equation
\be\label{ode}
g'(t) = \frac{S_x(t, g(t))}{4\pi\,g(t)} - 1, \quad t\in(0,T_0],\qquad
g(0) = R,
\en
where the function
\be\label{S}
S_x(t, g(t)):=\frac{\cos(\omega t)}{u(x,t+g(t))},\quad t>0
\en
is uniquely determined by the wave field measured at the receiver $x\in\Gamma_R$. The equation \eqref{ode}
characterizes a relation between
 the radial speed  of the moving source at $t>0$ and the causal signal $u(x, t+g(t))$. \tcr{Note that we have the upper bound $t+g(t)<T_0+R+R_0$ and
 by (\ref{ru}), $|S_x(t, g(t))|<8\pi (R+R_0)$ and for all $t\in(0,T_0)$}.

To investigate the well-posedness of (\ref{ode}), we introduce the function
\ben
F(t,\tau)=\frac{S_x(t, \tau)}{4\pi\,\tau} - 1,\qquad (t,\tau)\in \tcr{\mathcal{D}:=\{[0,T_0]\times [R-R_1, R+R_1]\}}.
\enn
Combining (\ref{S}) and (\ref{ru}), we have
\ben
S_x(t,\tau)=\frac{\cos(\omega t)}{u(x,t+\tau)},
\enn
and
\ben
u(x,t)=\frac{\cos(\omega b(t))}{4\pi (t-b(t))}\frac{1}{f'(b(t))},\quad t>R,\quad b(t):=f^{-1}(t).
\enn
Note that $t\neq b(t)$ and $f'(t)\neq 0$ for all $t>0$. \tcr{This implies the expression
\ben
S_x(t,\tau)=4\pi\frac{\cos(\omega t)}{\cos(\omega b(t+\tau))} \, (t+\tau-b(t+\tau))\, f'(b(t+\tau)).
\enn
Here we restrict the variables $(t,\tau)$ to a subset of $\mathcal{D}$:
\ben
(t,x)\in \mathcal{D}^*:=\mathcal{D}\cap \{(t,\tau): t+\tau>R, \quad |S_x(t,\tau)|<8\pi (R+R_0)\}.
\enn}
Since the orbit function $a(t)$ is of $\mathcal{C}^2$-smooth, the function $b(t)\in \mathcal{C}^2$ and $f'\in \mathcal{C}^1$. This implies that the function $(t,\tau)\rightarrow S_x(t,\tau)$ is $\mathcal{C}^1$-smooth on \tcr{$\overline{\mathcal{D}^*}$}. Further, one can prove that
\ben
\frac{{\rm d}\,F(t,\tau)}{{\rm d}\,\tau}=\frac{1}{4\pi}\left\{
-\frac{S_x(t,\tau)}{\tau^2}+\frac{1}{\tau}\frac{{\rm d}\,S_x(t,\tau)}{{\rm d}\,\tau}
\right\}\leq L \quad \tcr{\mbox{for all}\quad (t,\tau)\in\overline{\mathcal{D}^*}}.
\enn
Hence, the dynamical system \eqref{ode} admits a unique solution \tcr{in $\mathcal{D}^*$}. This implies that the distance function $|x-a(t)|$ for $0<t<T_0$ can be uniquely determined by $u(x,t)$ for $t\in(0,T)$ where $T=T_0+R+R_1$. Hence, the orbit function  $\{a(t): t\in(0,T_0)\}$ is uniquely determined by the wave fields $\{u(x^{(j)},t): j=1,2,3,4, t\in(0,T)\}$ detected at four receivers $x^{(j)}$ which do not lie on a plane. \end{proof}

\begin{rema}
The proof of Theorem \ref{th6} implies that for each $t_0>0$, we can get the distance $|a(t_0)-x^{(j)}|=g_j(t_0)$, where $g_j(t)$ solves
the equation (\ref{ode}) with $x=x^{(j)}\in \Gamma_R$, $j=1,2,3,4$.
This automatically gives an inversion scheme for calculating $a(t)$.
\end{rema}

\subsection{Source terms independent of one spatial variable}\label{sub3.2}

In this subsection we consider an inhomogeneous source term which does not depend on one spatial variable. Without loss of generality we suppose that $F(x,t)=\tilde{f}(\tilde{x},t)h(x_3)$, where
the function $\tilde{f}$ is compactly supported in $\tilde{B}_{R_0} \times [0, T_0]$ and $h$ is supported in $(-R_0, R_0)$ for some $R_0<R/\sqrt{2}$.
Here $\tilde{x}:=(x_1, x_2)$ and $\tilde{B}_{R_0}:=\{\tilde{x}\in \mathbb R^2| ~ |\tilde{x}|<R_0\}$. Our aim is to recover $\tilde{f}$, assuming that $h\in L^1(\textcolor{rot}{\mathbb R})$ is known in advance.
In particular, $\tilde{f}(\tilde{x},t)$ can be a moving source with the orbit lying on the $ox_1x_2$-plane and $h(x_3)$ can be regarded as a function approximating the delta distribution $\delta(x_3)$. Now, we consider the wave equation
\begin{align}\label{eqn4}
\begin{cases}
\textcolor{rot}{\partial^2_t}u(x,t) - \triangle u(x,t) = \tilde{f}(\tilde{x},t)h(x_3) \quad &\text{in} ~ \mathbb R^3 \times \mathbb R^+,\\
u(x,0) = \partial_t u(x,0) = 0 \quad &\text{in}~\mathbb R^3.
\end{cases}
\end{align}
%The main result of the inverse problems can be stated as follows. The inverse problem can be stated as follows.
%\textbf{Inverse Problem 4} (IP4): Let $D=\emptyset$. Assume that $f$ takes the form as in \eqref{eqn4} and $\text{supp}h \subset (-R_0, R_0)$ with $h$ non-uniformly vanishing. Determine the time and space dependent function $\tilde{f}$ from the radiated wave field $u$ measured on the surface $\Gamma_R \times (0,T_1)$ with $T_1>0$ sufficiently large.
%The uniqueness result can be stated as follows.
\textcolor{rot}{Throughout this subsection, the symbol $\,\widehat{\cdot}\,$ will denote the Fourier transform with respect to the time variable $t$.}
\begin{theo}\label{th4}
Assume that $h\neq 0$ is given. Then $\tilde{f}(\tilde{x},t)$ can be uniquely determined by $\{u(x,t): x\in \Gamma_R, t\in (0,T)\}$, where $T_1 = T_0 + R + R_0$.
\end{theo}
\begin{proof}
It suffices to prove that $\tilde{f}(\tilde{x},t) = 0$ if $u(x,t)=0$ for $x\in\Gamma_R$ and $t\in(0,T)$.
By the strong Huygens' principle, it holds that $u(x,t) = 0$ for \tcr{$|x|<R$} and $t> T$ (see \cite{HLLZ}). Then, applying the Fourier transform in time to $u$ in \eqref{eqn4} yields
\begin{align}\label{eqn4.1}
\begin{cases}
\triangle \hat{u}(x,\kappa) + \kappa^2 \hat{u}(x,\kappa) = \hat{\tilde{f}}(\tilde{x},\kappa)h(x_3) \quad &\text{in} ~ B_R,\\
\hat{u}(x,\kappa) = \partial_{\nu}\hat{u}(x,\kappa) = 0 \quad &\text{on}~\Gamma_R,
\end{cases}
\end{align}
where the Fourier transform of $u(x,t)$, given by
\[
 \hat{u}(x,\kappa) = \int_{B_R}u(x,t)\textcolor{rot}{{\rm e}}^{{\rm -i}\kappa t}{\rm d}t, \quad \kappa >0,
\]
satisfies the Sommerfeld radiation condition for any $\kappa>0$ (see \cite{HLLZ}). Here $\hat{\tilde{f}}(\tilde{x},\kappa)$ denotes the Fourier transform of $\tilde{f}(x,t)$. Define the test functions
\[
\varphi(x;\kappa_1) := \textcolor{rot}{{\rm e}}^{{\rm i}\kappa_1\tilde{x}\cdot \tilde{d}}\textcolor{rot}{{\rm e}}^{\sqrt{\kappa_1^2 - \kappa^2}x_3},\quad \tilde{d}\in \R^2,\quad |\tilde{d}|=1,\quad \kappa_1>\kappa.
\]
Then it is easy to verify that $\varphi$ satisfies the Helmholtz equation
\[
\triangle \varphi + \kappa^2 \varphi =0\quad\mbox{in}\quad \R^3.
\]
Multiplying both sides of \eqref{eqn4.1} by $\varphi$ and using integration by parts over $B_R$ yield
\[
\int_{B_R}\hat{\tilde{f}}(\tilde{x},\kappa)h(x_3)
\varphi(x){\rm d}x = \left(\int_{\tilde{B}_R}\hat{\tilde{f}}(\tilde{x},\kappa)\textcolor{rot}{{\rm e}}^{{\rm i}\kappa_1\tilde{x}\cdot \tilde{d}}{\rm d}\tilde{x}\right)
\left(\int_{-R}^R h(x_3)\textcolor{rot}{{\rm e}}^{\sqrt{\kappa_1^2 - \kappa^2}x_3}{\rm d}x_3\right) = 0.
\]
Since $h$ does not vanish identically, for $\kappa>0$ we can always find an interval $I$ such that  $\int_{-R}^R h(x_3)\textcolor{rot}{{\rm e}}^{\sqrt{\kappa_1^2 - \kappa^2}x_3}{\rm d}x_3 \neq 0$ for all $\kappa_1\in I$ and $\kappa_1>\kappa$, implying that \begin{align}\label{ft}
\int_{\tilde{B}_R}\hat{\tilde{f}}(\tilde{x},\kappa)\textcolor{rot}{{\rm e}}^{{\rm i}\kappa_1\tilde{x}\cdot d}{\rm d}\tilde{x} = 0
\end{align} for such $\kappa_1$.
Given $f(\tilde{x},t)$, denote by $\mathcal{F}(f)(\xi)$ \tcr{($\xi\in\R^3$)} the Fourier transform of $f$ with respect to the variable $(\tilde{x},t)\in \mathbb R^3$, i.e.,
\[
\mathcal{F}(f)(\xi) = \int_{\mathbb R^3}f(\tilde{x},t)\textcolor{rot}{{\rm e}}^{-{\rm i}\xi \cdot (\tilde{x},t)}{\rm d}\tilde{x}{\rm d}t,\quad \xi\in\R^3.
\]
Then the relation \eqref{ft} gives that
\[
\mathcal{F}(f)(\kappa_1\tilde{d},\kappa) = 0
\]
for all $\kappa_1 > \kappa >0$ and $|\tilde{d}|=1$. Since $\mathcal{F}(f)$ is analytic in $\mathbb R^3$ and $\{(\kappa_1\tilde{x},\kappa)|\kappa_1>\kappa, ~ |\tilde{d}|=0\}$ is an open set in $\mathbb R^3$, we have $\mathcal{F}(f)(\xi) = 0$ for all $\xi \in \mathbb R^3$, leading to $f(x,t)=0$. The proof is complete.
\end{proof}

\begin{rema}
Let $\tilde{f}(x,t) = \tilde{f}(\tilde{x}- \tilde{a}(t))$ be a moving source with the orbit $\tilde{a}(t):[0, +\infty)\rightarrow \tcr{\tilde{B}_R}$ lying on the $ox_1x_2$-plane. The proof of Theorem \ref{th4} implies the unique determination of the orbit $\tilde{a}(t)$. We refer to \cite{HKLZ} for more discussions concerning  inverse moving source problems in
electromagnetism.
\end{rema}
Based on the uniqueness proof of Theorem \ref{th4},
one can obtain a log-type stability estimate under strong a priori assumptions of $\tilde{f}$ and $h$. \tcr{The proof for the more complicated elastodynamical system was carried out in \cite{HK}. Below we only formulate the stability result and omit the proof for simplicity.}

\begin{theo}\label{th5}
Let $R>\sqrt{2}R_0, T_1=T_0+R+R_0$ and suppose $\tilde{f}\in H^3(\mathbb R^2\times R^+)\cap H^4(0,T; L^2(\mathbb R^2))$ satisfies
\[
\tilde{f}(\tilde{x},0) = \partial_t \tilde{f}(\tilde{x},0)  = \partial^2_t \tilde{f}(\tilde{x},0) = \partial^3_t \tilde{f}(\tilde{x},0) = 0, \quad \tilde{x}\in \mathbb R^2.
\]
Assume also that $h$ is non-uniformly vanishing with a constant sign $(h\geq 0~ \text{or}~ h\leq 0)$ and that there exists $M>0$ such that
\[
\|\tilde{f}\|_{H^3(\mathbb R^2 \times \mathbb R)} + \|\tilde{f}\|_{H^4(0,T;L^2(\mathbb R^2))}\leq M.
\]
Then, there exists $C>0$ depending on $M, R, T, \|h\|_{L^1(\mathbb R)}$ such that
\begin{equation*}
\|\tilde{f}\|_{L^2((0,T)\times\tilde{B}_{R})}\leq  C\left(\|u\|_{H^3(0,T;H^{3/2}(\partial B_{R}))}+\left|\ln\left(\|u\|_{H^3(0,T;H^{3/2}(\partial B_{R}))}\right)\right|^{-1}\right).\end{equation*}
\end{theo}

\section{Concluding remarks}\label{conclude}
This paper is mainly concerned with a Fourier-Laplace approach to inverse acoustic source problems using boundary dynamical  data over an infinite time interval. \textcolor{rot}{In situations where the Huygens' principle does not hold (e.g., the inhomogeneous background medium considered in Section \ref{ip1}), we apply the Laplace transform in place of the Fourier transform. The Fourier transform was used in the proof of Theorem \ref{th4}.}
It is worthwhile to investigate the uniqueness of recovering obstacles and source terms simultaneously using the data over a finite time interval \tcr{without any other assumptions on the source term at $t=0$}. This seems to be more realistic, but our approach of applying the Laplace transform cannot be applied.  The increasing stability issue for time-domain inverse source problems with respect to exciting frequencies would be interesting. However,  existing results are all justified in the time-harmonic regime only. The stability results in the time-domain  will provide deep insights into the resolution analysis of inverse scattering problems modeled by hyperbolic equations. Finally,  radiating and non-radiating time-dependent sources deserve to be rigorously characterized and classified. We hope to be able to address these issues and report the progress in the future.

\section*{Acknowledgement}
The work of G. Hu is supported by the NSFC grant (No. 11671028) and NSAF grant (No. U1530401). \tcb{The work of Y. Kian is supported by  the French National Research Agency ANR (project MultiOnde) grant ANR-17-CE40-0029.} The authors would like to thank Gen Nakamura for pointing the paper \cite{Isakov08} and for helpful discussions.

\end{document}